\documentclass[12pt,twoside, a4paper]{article}

\usepackage{latexsym}
\usepackage[english]{babel}
\usepackage{t1enc}
\usepackage{mathrsfs}
\usepackage{ifthen}
\usepackage{url}
\usepackage{epsfig}
\usepackage{amsbsy}
\usepackage{extarrows}
\usepackage{amsfonts}
\usepackage{amsmath,amssymb}
\usepackage{bbm}
\usepackage{amssymb}
\usepackage{amsthm}
\usepackage{amsxtra}
\usepackage{bbm}
\usepackage{color}
\usepackage{stmaryrd}
\usepackage{wasysym}
\usepackage{tikz}

\newcommand{\titel}{Planar Novikov–Shubin invariants for adjacency matrices of structured directed dense random graphs} 

\usepackage[margin=1.9cm]{geometry}

\usepackage{verbatim}

\usepackage[colorlinks=true,linkcolor=blue,anchorcolor=black,citecolor=black,filecolor=black,menucolor=black,runcolor=black,urlcolor=black]{hyperref} 

\numberwithin{equation}{section}
\numberwithin{figure}{section}

\newtheoremstyle{thm-style-oskari}
{7pt}      
{7pt}      
{\itshape} 
{}         
{\scshape} 
{.}        
{.5em}     
{}         

\theoremstyle{thm-style-oskari}
    \newtheorem{theorem}{Theorem}[section]
    \newtheorem{proposition}[theorem]{Proposition}
    \newtheorem{corollary}[theorem]{Corollary}
    \newtheorem{lemma}[theorem]{Lemma}
    \newtheorem{definition}[theorem]{Definition}
    \newtheorem{example}[theorem]{Example}
    \newtheorem{convention}[theorem]{Convention}
    \newtheorem{remark}[theorem]{Remark}

\newcommand{\bels}[2] {
        \begin{equation} \label{#1} \begin{split} 
                #2 
        \end{split} \end{equation}
        }
\newcommand{\bes}[1]{
        \begin{equation*}  \begin{split} 
                #1 
        \end{split} \end{equation*}
        }

\definecolor{olivegreen}{rgb}{0,0.6,0.1}
\newcommand{\nc}{\normalcolor}

\newcommand{\cedit}{}

\newcommand{\bbm}{\mathbbm} 
\renewcommand{\cal}{\mathcal} 
 
\renewcommand{\frak}{\mathfrak} 
\newcommand{\ol}[1]{\overline{#1} \!\,} 
\newcommand{\wh}{\widehat}
\newcommand{\wt}{\widetilde}

\newcommand{\eps}{\varepsilon}

\renewcommand{\P}{\mathbb{P}}
\newcommand{\E}{\mathbb{E}}
\newcommand{\R}{\mathbb{R}}
\newcommand{\C}{\mathbb{C}}
\newcommand{\N}{\mathbb{N}}

\newcommand{\ee}{\mathrm{e}} 
\newcommand{\ii}{\mathrm{i}} 
\newcommand{\dd}{\mathrm{d}}

\newcommand{\p}[1]{({#1})}
\newcommand{\pb}[1]{\bigl({#1}\bigr)}
\newcommand{\pB}[1]{\Bigl({#1}\Bigr)}
\newcommand{\pbb}[1]{\biggl({#1}\biggr)}

\newcommand{\cb}[1]{\bigl\{{#1}\bigr\}}
\newcommand{\cB}[1]{\Bigl\{{#1}\Bigr\}}
\newcommand{\cbb}[1]{\biggl\{{#1}\biggr\}}

\newcommand{\abs}[1]{\lvert #1 \rvert}

\newcommand{\norm}[1]{\lVert #1 \rVert}

\newcommand{\avg}[1]{\langle #1 \rangle}
\newcommand{\avgb}[1]{\big\langle #1 \big\rangle}
\newcommand{\avgB}[1]{\Big\langle #1 \Big\rangle}

\newcommand{\scalar}[2]{\langle{#1} \mspace{2mu}, {#2}\rangle}

\newcommand{\scalarbb}[2]{\bigg\langle{#1} \,\mspace{2mu},\, {#2}\bigg\rangle}

\DeclareMathOperator{\tr}{Tr}
\DeclareMathOperator{\var}{Var}

\DeclareMathOperator{\im}{Im}

\DeclareMathOperator*{\spec}{Spec}						

\newcommand{\1} {\mspace{1 mu}}
\newcommand{\2} {\mspace{2 mu}}

\newcommand{\mfour}[4]
{
\left(
\begin{array}{cccc}
#1 
\\
#2
\\
#3
\\
#4
\end{array}
\right)
}

\makeatletter
\def\blfootnote{\xdef\@thefnmark{}\@footnotetext}
\makeatother

\begin{document}

\blfootnote{Date: \today}
\blfootnote{MSC2010 Subject Classifications: 60B20, 15B52. }

\title{\vspace{-0.8cm}{\textbf{\titel}}} 
\author{
\begin{tabular}{c} Torben Kr\"uger\thanks{Financial support from VILLUM FONDEN  (Grant No. 10059)  is gratefully acknowledged.  \newline Email: \href{mailto:torben.krueger@fau.de}{torben.krueger@fau.de}}
\\ {\small 
\begin{tabular}{c} 
Friedrich-Alexander-Universität 
\\
Erlangen-Nürnberg
\end{tabular}
} \end{tabular}
\hspace*{0.5cm} \and \hspace*{0.5cm} 
\begin{tabular}{c}David Renfrew
\thanks{
 Email: \href{mailto:renfrew@binghamton.edu}{renfrew@math.binghamton.edu} 
} \\ {\small Binghamton University} 
\end{tabular} }
\date{}

\maketitle
\thispagestyle{empty} 

\begin{abstract}
The Novikov–Shubin invariant associated to a graph provides information about the accumulation of eigenvalues of the corresponding adjacency matrix close to the origin. For a directed graph these eigenvalues lie in the complex plane and 
having a finite value for the planar Novikov–Shubin invariant indicates a polynomial behavior of the eigenvalue density as a function of the distance to zero.  We provide a complete description of these invariants for 
  dense random digraphs with constant batch sizes, i.e. for  the directed stochastic block model.  The invariant depends only on which batches in the graph are connected by non-zero edge densities. We present an explicit finite step algorithm for their computation. For the proof we identify the asymptotic spectral density with the distribution of a $\C^K$-valued circular element in operator-valued free probability theory. We determine the spectral density in the bulk regime by solving the associated Dyson equation and infer the singular behavior of this density close to the origin by determining the exponents associated to the power law with which the resolvent entries of the adjacency matrix that corresponds to the individual batches diverge to infinity or converge to zero. 
\end{abstract}


\section{Introduction}
The stochastic block model is a random graph meant to model inhomogeneities in a disordered network. The vertices of the graph are partitioned into $K$ batches $V_1, \dots, V_K$. {\cedit A vertex} in $V_i$ is connected to {\cedit a vertex} in $V_j$ with probability $p_{ij}$ and these connections are independent. The stochastic block model can either be undirected or directed.  The undirected model with a single batch is called a dense Erd{\H o}s–R\'enyi model. In this work, we consider the directed stochastic block model and for notational simplicity we restrict to  constant batch sizes, i.e. $|V_i|=n$ for each $i=1, \dots, K$.  The undirected model with $K>1$ is  the dense inhomogeneous Erd{\H o}s–R\'enyi model (IERM) and converges in the limit $n \to \infty$ to a graphon in the cut distance \cite{Lovsz2004LimitsOD}. There is an extensive literature on this model and its sparse versions (see e.g. \cite{Bollobs2005ThePT,BORGS20081801,books/daglib/0031021,PhysRevE.66.066121} and references therein). An alternative description of the $n \to \infty$ limit for the dense IERM is provided by operator-valued free probability theory. The centered and appropriately normalized adjacency matrix $H_n:=\frac{1}{\sqrt{n}}(A_n - \E\1 A_n)$ of {\cedit this} undirected stochastic block model  converges to a $\C^K$-valued semicircular element $\frak{s} = \frak{s}^*$ in the sense of non-commutative distributions. \cedit Indeed, the limit of the empirical spectral measure  of $H_n$ is  determined  by a vector Dyson equation (see e.g. \cite{AEK2}). This equation is precisely the one satisfied by the Cauchy-transform of $\frak{s}$ in operator valued free probability theory \cite[Chapter 9]{MingoSpeicher}. 
 In particular, the convergence $\frac{1}{n} \tr p(H_n) \to \tau(p(\frak{s}))$ holds almost surely, where $\tau$ denotes the faithful, tracial state on the underlying $W^*$-probability space, $\cal{A}$, with $\frak{s} \in \cal{A}$ and $p$  a polynomial. 

The centered and normalized adjacency matrix $X_n:= \frac{1}{\sqrt{n}}(A_n - \E\1 A_n)$ of the directed stochastic block model (DSBM) is a random matrix with independent entries of differing variances. \nc It was shown in \cite{Altinhomogeneous, Cookvarprof1} that for such matrices the empirical spectral distribution (ESD) converges to a deterministic limit that corresponds to the Brown measure of the  $\C^K$-valued circular element $\frak{c}$, as the dimension tends to infinity. This measure, which was introduced in \cite{Brownmeasure} as an analog to the spectral measure for  non-Hermitian operators, is radially symmetric, supported on a disk and has a density $\sigma$ with respect the Lebesgue measure away from the origin. 
{\cedit Since the rank of $\E\1 A_n$} is of bounded order and in particular of lower order than $n$, the ESD of $\frac{1}{\sqrt{n}}A_n$ converges to the same limit (see Proposition~\ref{prop:Global law for A}).  {\cedit We note that an additive deformation whose rank is order $n$ rank} will, in general, break the radial symmetry and therefore change the limit. The addition of such deformations to i.i.d. matrices are  considered e.g. in \cite{Khoruzhenko1996,Sniady2002,BordenaveCapitaine2016,CampbellCipolloniErdosJi2024,ErdosJi2023} and to matrices with entries of differing variances in \cite{AK_Brown,AK_pseudo}. In the current work we focus on the local behavior of the density, $\sigma$, of the DSBM at the origin of the complex plane. 
Here the Brown measure has an atom if the matrix of connection probabilities $P=(p_{ij})_{i,j=1}^K$ does not have support, i.e. is not entry-wise bounded from below by a small positive constant times a permutation matrix, (see Appendix \ref{sec:Non-negative matrices} for information about non-negative matrices) and in this case $A_n$ is not invertible. In the case when $P$ has support, the Brown measure has a density  at the origin. However, this density may have a singularity at zero which can be interpreted as a measure of the invertibility of $A_n$. 
The Novikov–Shubin {\cedit invariant 
\bels{Novikov Shubin}{
c_{\rm NS}:=\lim_{t \downarrow 0}\frac{\textstyle{\log \int_{|z| \le t} \sigma(z) \dd z}}{\log t}\,.
}
is} a common way to quantify this density blow-up in Hermitian models that have power law growth, at zero, of the cumulative distribution function associated with the density $\sigma$. 
The invariant corresponds to the exponent in this power law \cite[Chapter 2]{Luck2002}. \cedit It was first introduced in \cite{NovikovShubin1986MorseIneqVonNeumann} and further developed in \cite{GromovShubin1991VonNeumannSpectraNearZero} to study spectral properties near zero of Laplace operators on Riemannian manifolds.  Since the Novikov–Shubin invariant quantifies the accumulation of eigenvalues at zero, it can be used to control large-time heat kernel decay, and thus gives global information for the manifold. \nc  In the Hermitian case, $\sigma$ is supported on the real line and the integral in \eqref{Novikov Shubin} is taken over $z \in \R$. 
 Since our directed model is non-Hermitian, we define the planar analog of the Novikov–Shubin invariant in exact analogy to \eqref{Novikov Shubin} with an integral over $z \in \C$.
Our main result shows that $c_{\rm NS}$ is finite for DSBM when $P$ has support and we provide an algorithm for computing this invariant in terms of the zero structure of $P$. In particular,  $c_{\rm NS}=c_{\rm NS}(P)$ only depends on the position of nonzero entries inside $P$.

The problem of invertibility of operators is of great interest and has many applications in Random Matrix Theory and Free Probability, \cite{mai2024fugledekadison,hoffmann2023computing,Arizmendi2024Universal,Shlyakhtenko2015Free,Rudelson2008Littlewood}.  {\cedit Novikov-Shubin invariant was} used in \cite{Shlyakhtenko2015Free} to prove the non-existence of atoms in noncommuting polynomials of free random variables, and in \cite{hoffmann2023computing} to give bounds on algorithm run times to prove invertibility of matrices in non-commuting elements, for the non-commutative Edmonds problem. In \cite{Kruegersingularity}, we considered Hermitian block random matrices and computed the severity of the blowup in the density, which follows a power law with rational exponent. {\cedit Concurrent with \cite{Kruegersingularity}, this blowup in the density was also obtained by O. Kolupaiev, \cite{Kolupaiev21}, for block variance profiles whose non-zero blocks are on and immediately above the anti-diagonal. Additionally in \cite{Kolupaiev21} the non-zero variances in the individual blocks are allowed to be non-constant with uniform bounds from above and away from zero. }

This result also proves that the Novikov–Shubin invariant for the undirected IERM  with associated connection probabilities $P$ is a positive rational number that only depends on the location of non-zero entries in $P$, whenever $P$ has support.

The proof in \cite{Kruegersingularity} relies on solving a discrete boundary value problem on the set of batch indices, whose solution on index $i$  provides the exponents with which the resolvent of the underlying matrix $H_n$, averaged {\cedit over the $i^{th}$ batch} and in the $n \to \infty$ limit, tends to zero or diverges to infinity as the spectral parameter approaches zero. In the current work we use Girko's Hermitization  \cite{Girko1985,Girko-book} to reduce the problem of determining the Brown measure of the non-Hermitian matrix $X_n$ to determining the spectral properties of a family of Hermitian matrices, indexed by the spectral parameter $z \in \C$, of twice the dimension of $X_n$ in a vicinity of the origin. Similar to \cite{Kruegersingularity} this leads to an equation for the exponents (cf. \eqref{non-herm min max averaging}), and the resolvent of this Hermitization of $X_n$, when averaged over the batches, diverges or approaches zero as the planar spectral parameter $z\in \C$ approaches the origin as given by the solution to this equation. Structurally, this equation is more complicated than the one encountered in  \cite{Kruegersingularity}  and does not have any given boundary values. In particular, contrasting the Hermitian setting, this equation does not, in general, admit a unique solution. It turns out, however, that we observe a cancellation in Girko's formula that makes this non-uniqueness not an obstacle for computing the behavior of the Brown measure $\sigma$.

\bigskip
\noindent
{\bf Notation}: We now introduce a few notations that are used throughout this work. We use the comparison relation $\varphi \lesssim \psi $ (or $\psi \gtrsim \varphi$) between two positive quantities $\varphi,\psi>0$ if there is a constant $C=C(S)>0$, only depending on the variance profile $S$, such that $\varphi \le C\2 \psi $. We write $\varphi \sim \psi $ in case $\varphi \lesssim \psi $ and $\varphi \gtrsim \psi $ both hold. When a vector is compared with a scalar, it is meant that relation holds in each component of the vector.  Furthermore, we identify constant vectors and scalars.  For vectors $f=(f_i)_i, g=(g_i)_i$ we interpret  $f \lesssim g $ entrywise, i.e. $f_i \lesssim g_i $ holds for all $i$. In general, we consider $\C^d$ as an algebra with entrywise operations, i.e. we write $\phi(f):=(\phi(f_i))_{i}$ for a function $\phi: \C \to \C$ applied to a vector $f$  and $fg:=(f_ig_i)_i$ for the product of two vectors. Our scalar products are normalized, meaning that for $f,g \in \C^d$, $\scalar{f}{g}:= \frac{1}{d}\sum_i\ol{f}_ig_i$ and we use the short hand $\avg{f}:=\scalar{1}{f}=\frac{1}{d}\sum_i f_i$ for the average of a vector.

\section{Main results}

Let $A = (a_{\alpha\beta}^{ij})_{i,j=1, \dots K}^{\alpha, \beta = 1, \dots ,n}\in  \R^{N \times N} = \R^{K \times K} \otimes \R^{n \times n}$ with $N = n \1K$ be the adjacency matrix of a dense directed stochastic block model, i.e. $a_{\alpha\beta}^{ij}$ are independent and ${\rm Bern}(p_{ij})$-distributed, {\cedit with $p_{ij} \in [0,1)$}, for $i,j =1, \dots, K$ and $\alpha, \beta =1, \dots, n$.  We set $S:=(s_{ij})_{i,j=1}^K $ with  $s_{ij}:= p_{ij}(1-p_{ij}) = \var(a_{\alpha\beta}^{ij})$ the variance profile of $A$.
We assume that $P=(p_{ij})_{i,j=1}^K$ and with it  ${S}$ is irreducible, as the spectrum of $A$ can be computed separately on each irreducible component.

The self-consistent density of states associated with $S$ was introduced in { \cite[(2.5)]{Altinhomogeneous} (see 
Remark 2.6 of \cite{Cookvarprof1} for an alternative description)} and  is a probability measure  on the complex plane with density $\sigma$ away from the origin, given by the identity
\bels{eq:defsig}{
\sigma(z) = -\frac{1}{2\pi}\Delta_z \int_0^\infty \pbb{\avg{v(\abs{z}^2, \eta)}- \frac{1}{1+\eta}} \dd \eta 
}
for $z \ne 0$, where $v=v(\tau, \eta) \in (0,\infty)^K$ {\cedit is the unique positive solution (see Remark \ref{Rem:unique solution}, below) to} the coupled systems of equations
\bels{Dyson equation with eta}{
\frac{1}{v} = \eta + Sw + \frac{\tau}{\eta + S^t v}\,, \qquad \frac{1}{w} = \eta + S^tv + \frac{\tau}{\eta + S w},
}
together with $w=w(\tau, \eta) \in (0,\infty)^K$, where $\eta, \tau>0$. The following proposition identifies $\sigma$ as the limiting spectral density associated to the bulk eigenvalues of $A$. There are up to $K$ eigenvalues of size $O(n)$ which do not contribute to the density $\sigma$ in the $n \to \infty$ limit. 
\begin{proposition}\label{prop:Global law for A}
Let $A $ be the adjacency matrix of the directed stochastic block model as defined above. Then 
\begin{equation}\label{eq:gloabal law for A}
\frac{1}{n\1K} \sum_{w \in \spec(A) }\delta_{\frac{1}{\sqrt{n}}w} \to \sigma( z)\dd^2 z  + \varrho\1 \delta_0
\end{equation} 
weakly in probability, where $\varrho = 1- \int_{\C} \sigma(z) \dd^2 z$. 
\end{proposition}

The proof of this proposition follows an argument that is by now well established and relies on Girko's Hermitization trick, that was first used on random matrices with i.i.d. entries \cite{Girko1985, Girko-book}. For the convenience of the reader we recall this argument in the appendix.

\begin{remark} \label{Rem:unique solution}
 The system of equations \eqref{Dyson equation with eta} is called the Matrix Dyson Equations (MDE) for the Hermitization of $\frac{1}{\sqrt{n}}(A-\E\1 A)$ which is a Hermitian random matrix of dimension $2 n \1K$ (see Appendix~\ref{sec:Girko's trick} for details on how it is used). Its solution is unique \cite{Helton01012007}.   In  \cite[Example 2.3]{Cookvarprof1}, it is shown that an irreducible $S$ satisfies the admissibility assumption \cite[Assumption A.2]{Cookvarprof1}.  This implies that \eqref{Dyson equation with eta} and its solution can be extended to $\eta =0$ and its solution is strictly positive for $\tau \in (0,\rho(S))$, where $\rho(S)$ denotes the spectral radius of $S$ \cite[Prop. 3.2]{Altinhomogeneous},\cite[Theorem 2.2]{Cookvarprof1}.  
 The $\eta \to 0$ limit of solution to \eqref{Dyson equation with eta} is also the positive unique solution to the equation with $\eta =0$, once the additional condition  $\langle v \rangle = \langle w \rangle $ is added.
In \cite[Lemma~4.1]{Altinhomogeneous}, it was shown that when the entries of $S$ are bounded from above and below, $\sigma(z)$ is positive and bounded from above on the disk with radius $\sqrt{\rho(S)}$ centered at the origin. This bound was extended to variance profiles, $S$, that are fully indecomposable \cite[Proposition 2.7 and Theorem 2.8]{Cookvarprof2}, where it was also shown that the density is strictly positive on its support away from $z=0$ for any irreducible $S$. 
\end{remark}

 If $S$ does not have support,  i.e. any multiple of $S$ is not bounded from below by some permutation matrix (see Appendix~\ref{sec:Non-negative matrices} for the precise definition),  then by Frobenius-K{\"o}nig theorem it has a large block of zeros. It was shown in \cite{Kruegersingularity}, for the analogous Hermitian model that this implies there is an atom at $0$ for the limiting spectral measure. By considering the Hermitization of our model, this result implies the Brown measure of the non-Hermitian model also has{ \cedit an atom at zero, i.e. $\varrho>0$} in \eqref{eq:gloabal law for A} and the Novikov-Shubin invariant in \eqref{Novikov Shubin} is infinite. 
In the following, we therefore assume that $S$ has support, {\cedit and will show that, under this assumption, $\varrho =0$}.

 Under the assumption that $S$ has support, by \cite[Theorem 4.2.6]{brualdi_ryser_1991}, there exist two permutation matrices, $Q_1,Q_2$,
that bring $S$ into the  block upper triangular form
\nc
\bels{def wt S}{
\wt{S}:=Q_1SQ_2^t = 
\left(
\begin{array}{cccc}
\wt S_{11} & &&\bf{\star}
\\
 & \wt S_{22} &&
 \\
 & & \ddots&
 \\
 \bf{0}&  && \wt S_{LL}
\end{array}
\right)
}
where $\wt S_{11}, \wt S_{22}, \dots, \wt S_{LL}$ are  primitive square matrices with positive main diagonal.  By Lemma~\ref{lmm:FID matrices} the matrices $\wt S_{ii}$ are fully indecomposable (FID), which means they do not have a zero submatrix  whose added row and column dimensions are equal to the dimension of $\wt S_{ii}$  (we recall this notion in Appendix~\ref{sec:Non-negative matrices}).   The identity \eqref{def wt S} provides an $L \times L$-block structure for $\wt{S} = (\wt S_{lk})_{l,k=1}^L$.

We denote $Q:= Q_1Q_2^t $ and introduce the relations on $\llbracket L \rrbracket$ given by:
\bels{def of orders}{
 l \LHD  k \quad :\Leftrightarrow  \quad  \wt S_{lk} \ne 0 \text{ and } l \ne k\,, \qquad l \prec  k \quad :\Leftrightarrow  \quad  Q_{lk} \ne 0\,,
}
where $Q_{lk}$ denotes the $(l,k)$ block inside $Q$ according to the block structure of $\wt{S}$ in \eqref{def wt S}.
 We then use this relation to define an edge labeled {\cedit directed} graph on $\llbracket L \rrbracket$, with edges labeled by $ \LHD $ and $\prec$.  
 The set of edges of $\#$-type on $\llbracket L \rrbracket$ we denote by  $\cal{E}^\#:=\{(l,k) \in \llbracket L \rrbracket^2 : l \# k\}$, where $\#= \LHD, \prec$.   
 A path $\gamma = e_1 e_2 \dots e_\ell$ is a finite concatenation of edges $e_l =(k_{l-1},k_{l}) \in \cal{E}^\LHD \cup \cal{E}^\prec$ and we denote $\abs{\gamma}:=\ell$ and $\avg{\gamma}:=\abs{\{l\in\llbracket \ell \rrbracket: e_l \in \cal{E}^\prec \}}$. We call $\gamma$ a cycle when $k_0 = k_{\ell}$. Note that since $S$ is irreducible the graph  $\llbracket L \rrbracket$  with edge set  $\cal{E}^\LHD \cup \cal{E}^\prec$ is strongly connected, i.e. for any two indices $l,k \in \llbracket L \rrbracket$ there is a path from $l$ to $k$. Indeed, this statement means that there is a product of the matrices $(\bbm{1}(l\LHD k))_{l,k=1}^L$ and $(\bbm{1}(l\prec k))_{l,k=1}^L$ such that the  $(l,k)$-entry is positive. By the definition of the edges, this holds true in particular if a product of powers of the matrices $Q$ and $\wt{S}^\circ$ has a non-zero $(l,k)$-block, where $\wt{S}^\circ$ is the $L \times L$ block matrix resulting from $\wt{S}$ by setting the diagonal blocks $\wt{S}_{ii}$ equal to zero. Instead we  can consider  products of the matrices $Q$ and $Q+Q\wt{S}^\circ$. 
 Since the blocks $\wt{S}_{ii}$ have a positive diagonal, the non-zero blocks of $Q+Q\wt{S}^\circ$ and $Q\wt{S}$ coincide and it is therefore sufficient to show that a product of $Q$ and $Q\wt{S}$ has a non-zero $(l,k)$-block.
That $Q$ is a permutations matrix implies  $Q^t$ is a power  of $Q$, and it is thus enough to show that  $Q^t \wt{S} = Q_2 S Q_2^t$  is irreducible, which is equivalent to the  irreducibility of $S$. This, however, is true by our assumptions.

Our main theorem shows that $\sigma$ has a singularity of severity $\abs{z}^{-2(1-\kappa)}$ at the origin. 
\begin{theorem}\label{thr:Singularity degree}  
Let $S$ be an irreducible non-negative matrix, with support. Let the minimal ratio of $\prec$-edges to {\cedit total edges}, as defined in \eqref{def of orders}, on cycles be 
\begin{equation}\label{eq:defkapp}
\kappa:= \min_{\gamma} \frac{\avg{\gamma}}{\abs{\gamma}} \in (0,1]
\end{equation}
where the minimum is taken over all cycles $\gamma$.
Then near the origin, the self-consistent density of states, $\sigma$, given in \eqref{eq:defsig}, blows up like:
\[
\lim_{z \to 0} \frac{\log \sigma(z)}{\log |z|} = -2+2\1\kappa\,.
\]
\end{theorem}

\begin{corollary}[Novikov–Shubin invariant] The planar Novikov–Shubin invariant, defined by \eqref{Novikov Shubin}, for  the directed stochastic block model with a matrix of connection probabilities $P$ that is irreducible and has support is $c_{\rm NS} = 2\kappa$ with $\kappa$ defined in  \eqref{eq:defkapp}. In particular, $c_{\rm NS} = 2 \frac{1+\ell_1}{1+\ell_1+\ell_2}$ for two non-negative integers $\ell_1, \ell_2$. 
\end{corollary}

We now briefly compare the results of \cite{Kruegersingularity}, in the Hermitian case, with the current results. In \cite{Kruegersingularity}, the first step was to conjugate $S$ by a single permutation matrix into a normal form, similar to \eqref{def wt S}, but preserving the Hermitian symmetry. From this normal form the degree of singularity in the density was computed from the length of the longest down left part through the anti-diagonal of the matrix. In particular, the  singularity degree is given as a rational function of a single natural number. In the current paper, we require two permutation matrices to bring $S$ into normal form and the degree of the singularity depends not just on the resulting normal form but also the interaction between this normal form and the product $Q$ of these two permutation matrices. This leads to a richer class of singularities, parametrized by two natural numbers.

\section{Behavior at the singularity of the solution to the Dyson Equation}
In this section we consider the limit of \eqref{Dyson equation with eta} when $\eta  \downarrow 0$ and $\tau \downarrow 0$. For $\eta=0$ the solution to \eqref{Dyson equation with eta} is no longer unique,  because if $(v,w)$ is a solution then so is $(\alpha\1v,\alpha^{-1}w)$ for any  $\alpha>0$. To resolve this issue, we consider only the solution that results from taking the $\eta  \downarrow 0$ limit of solutions at positive $\eta$. Since for $\eta>0$ the solution satisfies the equality $\langle v \rangle = \langle w \rangle$, this constraint is inherited   for $\eta=0$. Additionally, we assume $\tau \in (0, \rho(S))$, so the solution is strictly positive.

At $\eta=0$ with $\tau < \varrho(S)$ the Dyson equation takes the form
\bels{Dyson equation}{
 \frac{1}{v} = Sw + \frac{\tau}{S^tv}\,, \qquad \frac{1}{w} = S^tv + \frac{\tau}{Sw}\,,
}
which has a unique positive solution with the additional assumption that $\avg{v} =\avg{w}$, \cite[Theorem 2.2]{Cookvarprof1}. 
When replacing the  two  terms $ \frac{\tau}{S^tv},  \frac{\tau}{Sw}$ both with a constant $\eta>0$,  equations \eqref{Dyson equation}  describe the asymptotic behavior of the singular value density of $X$ in the vicinity of the origin. This setup was studied in \cite{Kruegersingularity}. Since in the current setup the two additional terms depend on  the solution itself the analysis of  equation \eqref{Dyson equation} is  much more delicate.

Equations \eqref{Dyson equation} are written equivalently as
 \bels{Dyson equation-version 2}{
 v = \frac{S^t v}{(Sw)(S^t v) + \tau}\,, \qquad   w = \frac{Sw}{(Sw)(S^t v) + \tau}\,.
 }
 As a consequence we see that 
 \bels{vw symmetry}{
 v Sw  = w S^t v\,.
 }
 In particular, this identity allows us to write \eqref{Dyson equation} in terms of $v$. Multiplying the second equation by $w$ gives
\bels{vw equation with 1 on lhs}{
 1 = w S^tv + \frac{vw\tau}{vSw} = w S^tv+ \frac{vw\tau}{w S^tv} =  w S^tv + \frac{v\tau}{S^tv} \,.
 }
 Therefore, we find 
 \[ w = \frac{1}{S^t v} \left(1 - \frac{v\tau}{S^tv}  \right)\,.\]
 Substituting this back into the first equation yields
 \bels{vDyson equation}{
 \frac{1}{v} =S\pB{\frac{1}{S^tv}\pB{1 - \frac{\tau v}{S^t v}}} + \frac{\tau}{S^t v}\,.
}
Additionally, because $w>0$, we have that $\left(1 - \frac{v\tau}{S^tv}  \right)>0$, which implies that $ 0 < v < \frac{1}{\tau} S^t v$.

 We now use the permutations from the normal form of $S$, \eqref{def wt S}, to permute the indices of $v$ and $w$ and let  $\wt{v}:= Q_1 v$,  $\wt{w}:=Q_2 w$ with $Q= Q_1Q_2^t $. Then \eqref{Dyson equation} is equivalent to  
\bels{tilde equations}{
 \frac{1}{\wt{v}} = \wt{S}\wt{w} + Q\frac{\tau}{\wt{S}^t\wt{v}}\,, \qquad \frac{1}{\wt{w}} = \wt{S}^t\wt{v} +Q^t \frac{\tau}{\wt{S}\wt{w}}\,,
}
the identities \eqref{Dyson equation-version 2} become 
\[
\wt{v} = \frac{Q\wt{S}^t\wt{v} }{(Q\wt{S}^t\wt{v})(\wt{S}\wt{w} ) + \tau}\,, \qquad Q\wt{w}= 
\frac{\wt{S}\wt{w} }{(Q\wt{S}^t\wt{v})(\wt{S}\wt{w} ) + \tau}\,,
\]
and \eqref{vw symmetry} takes the form 
\bels{vw symmetry with P}{
\wt{v}\wt{S}\wt{w} = Q(\wt{w}\wt{S}^t \wt{v})\,. 
}
We denote $\wt{v} =(\wt{v}_1, \dots, \wt{v}_L)$ and $\wt{w} =(\wt{w}_1, \dots, \wt{w}_L)$ according to the block structure of $\wt{S}$.

\begin{lemma} \label{lmm:complementary scaling}
There is a constant $C>0$, such that uniformly for all $\tau \le \rho(S)/2$ the polynomial bounds 
\bels{power law bound}{
\tau^C \lesssim v \lesssim \tau^{-C}\,,\qquad  \tau^C \lesssim w \lesssim \tau^{-C},
}
are satisfied. Furthermore,  
\bels{v average scaling}{
\wt{v}_k\sim \avg{\wt{v}_k} \sim \frac{1}{\wt{w}_k} \sim \frac{1}{\avg{\wt{w}_k}} \,, 
}
hold true for all $k \in \llbracket L \rrbracket$  and 
\bels{monotonicity for v}{
\avg{\wt{v}_k} \lesssim \avg{\wt{v}_l}\,, \qquad \avg{\wt{w}_l} \lesssim \avg{\wt{w}_k}\qquad \text{ for all }\; k \LHD l\,.
}
\end{lemma}
Before we prove this lemma, we  first provide a classification of the solution of the Dyson equation, \eqref{Dyson equation}, as the minimizer of the functional 
\bels{def of J}{
J: \bigg\{x \in \R^K : 0< x < \frac{1}{\tau} S^t x\bigg\} \to [1, \infty)\,, \qquad 
 J(x):= \avgB{\frac{\tau x}{S^t x}} - \avgB{\log \frac{\tau x}{S^t x}} 
 }
 for $\tau>0$. 
 Note that $J$ is invariant under scaling, i.e. $J(\alpha\1 x) = J(x)$ for all positive numbers $\alpha  >0$. The lower bound $J(x) \ge 1$ is a consequence of the fact that $\beta  - \log \beta \ge 1$ for all $\beta >0$.   We will see below that the directional derivate of this functional (see \eqref{eq:dir deriv J}) recovers the equation for $v$, i.e.  \eqref{vDyson equation}.

\begin{proposition}[Variational problem for $v$] \label{prp:variational problem for v} Up to rescaling, the solution, $v$, to the Dyson equation \eqref{Dyson equation} is the unique minimizer for $J$, i.e., $v$ lies in the domain of $J$ and for $ J_\ast:=\inf_{x} J(x)$ we have $J(v) = J_\ast$. Furthermore, any minimizing sequence $x^{(n)}$ with $\lim_{n \to \infty}J(x^{(n)}) = J_\ast$ and $\avg{x^{(n)}}=\avg{v}$ satisfies $x^{(n)} \to v$. 
\end{proposition}

\begin{proof}
We begin by showing that the accumulation points of any minimizing sequence $x^{(n)}$ for $J$ with $\avg{x^{(n)}}=\avg{v}>0$ lie inside the domain of $J$. First, we will show that
  $\inf_{n \in \N}x^{(n)}>0$. Indeed, suppose that for some $i \in \llbracket  K \rrbracket$ we have  $x^{(n)}_i \to 0$. Since 
 \[
\frac{ \tau x^{(n)}_i}{(S^tx^{(n)})_i} - \log \frac{ \tau x^{(n)}_i}{(S^tx^{(n)})_i}
 \]
 remains bounded for all $i \in \llbracket  K \rrbracket$, we conclude 
 \[
 \liminf_{n \to \infty} \frac{  x^{(n)}_i}{(S^tx^{(n)})_i}>0\,.
 \]
Therefore we have  $(S^tx^{(n)})_i \to 0$, i.e. $x^{(n)}_j \to 0 $ for all $j \in \{k : s_{ki}>0\}$. Continuing this argument shows that $x^{(n)}_j \to 0 $ for all indices $j$ such that $(\ee^{S})_{ij}>0$. Since $S$ is irreducible this is the case for all $j$. Thus, $x^{(n)} \to 0$, which contradicts $\avg{x^{(n)}}=\avg{v}$.

Now let $x^\ast$ be an accumulation point of the minimizing sequence
 with  $0<x^\ast \le \frac{1}{\tau}S^t x^\ast $. We show that $x^\ast < \frac{1}{\tau}S^t x^\ast$. Indeed suppose 
 $x^\ast_i = \frac{1}{\tau}(S^t x^\ast)_i$ for some $i \in \llbracket  K \rrbracket$. Then 
 \[
 \nabla_{e_i} J (x^\ast)= 
 \avgB{e_iS\pB{\frac{1}{S^t x^\ast}\pB{1-\frac{\tau x^\ast}{S^t x^\ast}}}}\,,
 \]
 where we extended $J(x)$ to all $x>0$. Here we used that the directional derivative of $J$ in the direction $h \in \R^K$ is 
  \begin{equation}\label{eq:dir deriv J}
 \nabla_h J(x) = 
 \avgB{h\pB{\frac{\tau }{S^t x}-S\frac{\tau x}{(S^t x)^2}+S\frac{1}{S^t x}-\frac{1}{x}}}\,.
 \end{equation}
Therefore we get the contradictory statement $J(x^\ast - \eps\1 e_i) <J_\ast = J(x^\ast)$ for sufficiently small $\eps>0$ unless $x^\ast_j = \frac{1}{\tau}(S^t x^\ast)_j$ for all $j  \in \{k : s_{ik}>0\}$. Again by continuing this argument and using that $S$ is irreducible, we get 
 $x^\ast = \frac{1}{\tau}S^t x^\ast$. But the spectral radius of $S$ is $\rho(S)$ and $\tau <\rho(S)$.  This contradicts the Perron-Frobenius Theorem.  
 
 We conclude that an accumulation point  $x=x^\ast$ with $\avg{x}=\avg{v}>0$ of the minimizing sequence lies in the domain of $J$. Thus it satisfies the equation 
  \bels{v equation}{
 \frac{1}{x}=S\pB{\frac{1}{S^tx}\pB{1 - \frac{\tau x}{S^t x}}} + \frac{\tau}{S^t x}\,, \qquad 0< x < \frac{1}{\tau} S^t x\,.
 }
 Now we show that this equation has a unique solution with $\avg{x} =\avg{v}$.  From \eqref{vDyson equation} we see that $v$ solves this equation,  ensuring existence. For the uniqueness, suppose $x$ solves \eqref{v equation} and $\avg{x} = \avg{v}$. Then we define 
  \[
 y:=\frac{1}{S^tx}\pB{1 - \frac{\tau x}{S^t x}}>0\,.
 \]
 By definition of $y$ and \eqref{v equation} we get
   \bels{x y equation}{
 \frac{1}{x} = Sy + \frac{\tau}{S^t x}\,, \qquad 
 \frac{1}{y} = \frac{(S^tx)^2}{S^t x-\tau x}= S^tx +\frac{\tau x S^tx}{S^t x-\tau x} =S^tx +\frac{\tau }{Sy} \,,
  }
   where the last identity holds because the first equation in \eqref{x y equation} implies 
  \[
  \frac{S^t x-\tau x}{ x S^tx}= \frac{1}{x} - \frac{\tau}{S^tx } = Sy\,.
  \]
 With the replacement $(x,y) \to (v,w)$ we see that $(x,y)$ solves the Dyson equation \eqref{Dyson equation}. But this equation has a unique solution, up to the rescaling $(v,w) \to (\alpha \1 v, \frac{1}{\alpha} w)$. Thus, \eqref{v equation} has a unique solution with the normalization $\avg{x}=\avg{v}$. Altogether this finishes the proof of the proposition. 
\end{proof}

\begin{proof}[Proof of Lemma~\ref{lmm:complementary scaling}] 
By Proposition~\ref{prp:variational problem for v}, we know  $v$ minimizes the functional $J$ from \eqref{def of J}. In particular, with $s>0$ being the Perron-Frobenius eigenvector for $S^t$, i.e. $S^ts=\rho(S)s$ we see that 
\bels{upper bound on Jv}{
J(v) \le J(s) = \frac{\tau}{\rho(S)} -\log \frac{\tau}{\rho(S)}\,.
}
Since $\inf_{\beta >0}(\beta - \log\beta) = 1$ we infer the first inequality of 
\[
 K-1 -\log \frac{ \tau v_i}{(S^tv)_i} \le K\1J(v) \le K\1\pB{ \frac{\tau}{\rho(S)} -\log \frac{\tau}{\rho(S)}} \le K\pB{1 -\log \frac{\tau}{\rho(S)}}
\]
 for any index $i \in \llbracket  K \rrbracket$. We conclude 
 \[
{\cedit S^tv \le  \ee  \2\pbb{\frac{\tau}{\rho(S)}}^{-K} \tau v\,. }
 \]
 Together with $\tau \1v \le S^tv$ and the irreducibility of $S$ we get the polynomial upper and lower bound on $v$ in  \eqref{power law bound}. Since $\avg{v} = \avg{w}$ this also implies an upper bound on $w$ of the form $w \lesssim \tau^{-C}$. With \eqref{Dyson equation-version 2} we also get $w \gtrsim \tau^{C}Sw$ by potentially increasing the value of the constant $C>0$. Since $S$ is irreducible this implies the polynomial  bounds on $w$ in  \eqref{power law bound}. 
 
 Now we show \eqref{v average scaling}. From \eqref{upper bound on Jv} and the definition of $J$ we see that 
\bels{log permutation bound}{
- \avgB{\log \frac{ \wt{v}}{\wt{S}^t \wt{v}}}=-\avg{\log \wt{v}} + \avgb{\log \wt{S}^t \wt{v}}=- \avgB{\log \frac{ v}{S^t v}} \le \tau\,,
}
by permutation of the indices. Since $\wt{S}$ has positive entries along its diagonal, we have $\wt{S}^t \wt{v} \gtrsim \wt{v}$. \cedit Here we used that $S$ has support.\nc With \eqref{log permutation bound} this also implies the bound 
$\frac{ \wt{v}}{\wt{S}^t \wt{v}} \gtrsim 1 $. By defining a functional 
\[
I: \bigg\{{\cedit y} \in \R^K : 0< y < \frac{1}{\tau} S y\bigg\} \to [1, \infty)\,, \qquad 
 I(y):= \avgB{\frac{\tau y}{S y}} - \avgB{\log \frac{\tau y}{S y}}
\]
analogous to $J$ from \eqref{def of J} and exchanging the roles of $v$ and $w$, as well as $S$ and $S^t$, we see that a variational principle for $w$ holds analogously to Proposition~\ref{prp:variational problem for v}. Following the same reasoning as we did for $v$ then yields that
\bels{wtS w sim w}{
\wt{S}^t \wt{v} \sim \wt{v} \,, \qquad \wt{S} \wt{w} \sim \wt{w}
}
also holds for $w$. We multiply the first equation in \eqref{tilde equations} with $\wt{v}$ to find 
\bels{Permuted equation with 1}{
1= \wt{v}\wt{S}\wt{w} + \frac{\tau\wt{v}\wt{w}}{\wt{w}Q\wt{S}^t\wt{v}} \sim \wt{v}\wt{w} + \frac{\tau\wt{v}\wt{w}}{\wt{w}Q\wt{v}}\,.
}
In particular, we see the upper bound  $\wt{v}\wt{w} \lesssim 1$. By \eqref{vw symmetry with P} we have $\wt{v}\wt{w} \sim Q(\wt{v}\wt{w})$. Applying the permutation $Q$ iteratively reveals that $(\wt{v}\wt{w})_i \sim(\wt{v}\wt{w})_j $ for all $i,j \in \llbracket K \rrbracket$
 that are elements of the same cycle of $Q$. 
Now we show the lower bound $\wt{v}\wt{w} \gtrsim 1$. Suppose that this bound fails, i.e. that $(\wt{v}\wt{w})_i\ll 1$ for some index $i\in \llbracket K \rrbracket$. Then the same behavior holds for all indices $i$ in the same cycle of $Q$. We call these indices $I_Q \subset \llbracket K \rrbracket$ and write $x_{I_Q} =(x_i)_{i \in I_Q}$ for the restriction of $x \in \C^K$ to these indices. 
By \eqref{Permuted equation with 1} and $\wt{v}\wt{w} \sim Q(\wt{v}\wt{w})$ we get 
\[
1\sim \wt{v}\wt{w} + \frac{\tau\wt{v}Q\wt{w}}{\wt{v}\wt{w}}\,.
\]
Since we assumed that $(\wt{v}\wt{w})_{I_Q} \ll 1$, we must have $(\wt{v}\wt{w})_{I_Q} \sim \tau(\wt{v}Q\wt{w})_{I_Q}$,  i.e. $\wt{w}_{I_Q} \sim \tau (Q\wt{w})_{I_Q}$.  Now we iterate  this relation which leads to a contradiction, since for some $n \in \N$ we have $Q^n = 1$ and would thus conclude  $\wt{w}_{I_Q} \sim \tau^n \wt{w}_{I_Q}$, which is not true for $\tau \ll 1$. We infer $ \wt{v}\wt{w} \sim 1$. 

Recalling that $\wt S_{kk}$  with $k \in \llbracket L \rrbracket$ in the block structure \eqref{def wt S} of $\wt{S}$ are primitive we find $n \in \N$ such that $\wt S_{kk}^n$ has strictly positive entries and get
\[
\wt{w}_k \sim (\wt{S}^n\wt{w})_k  \gtrsim \wt S_{kk}^n\wt{w}  \sim \avg{\wt{w}_k}  
\]
due to \eqref{wtS w sim w}. In combination with the trivial upper bound $\wt{w}_k \lesssim \avg{\wt{w}_k}$, we conclude $\wt{w}_k \sim \avg{\wt{w}_k}$. For $v$ we prove $\wt{v}_k \sim \avg{\wt{v}_k}$ analogously. Together with $\wt{v} \wt{w} \sim 1$ \eqref{v average scaling} follows. By definition of the relation $\LHD$ in \eqref{def of orders} and \eqref{wtS w sim w} the monotonicity \eqref{monotonicity for v} follows. 
\end{proof}

 We now use Lemma \ref{lmm:complementary scaling} to show that $\langle \wt v_k \rangle$ satisfies a min-max averaging type equation, given in \eqref{comparison relation vi}, on a graph induced by the edges $\cal{E}^\prec$ and $\cal{E}^\LHD$, defined in \eqref{def of orders}. From this min-max averaging type equation we will determine the degree of the blow-up of $\sigma$. 

We write \eqref{tilde equations} in the form 
\bels{v w equations with 1 on the lhs}{
1=\wt{v}_k\sum_{l=1}^L \wt S_{kl} \wt{w}_l + a_k \wt{v}_k\,, \qquad 1=\wt{w}_k\sum_{l=1}^L \wt S_{kl}^t \wt{v}_l + b_k \wt{w}_k\,, \qquad k\in \llbracket L \rrbracket\,,
}
where we introduced the short hand $a=(a_1, \dots, a_L)$ and $b=(b_1, \dots, b_L)$ for  
\[
a := Q\frac{\tau}{\wt{S}^t\wt{v}}\,, \qquad b:= Q^t \frac{\tau}{\wt{S}\wt{w}}\,.
\]
We equate the two right hand sides in \eqref{v w equations with 1 on the lhs},  average them {\cedit within each block} and {\cedit then for each $k \in \llbracket L\rrbracket$,} cancel the term $\avg{\wt{v}_k\wt S_{kk} \wt{w}_k}$ to get 
\[
\sum_{l: l \ne k} \avg{\wt{w}_k\wt S_{lk}^t \wt{v}_l} + \avg{b_k\wt{w}_k} = \sum_{l: l \ne k} \avg{\wt{v}_k\wt S_{kl} \wt{w}_l} + \avg{a_k \wt{v}_k} \,.
\]
By definition of the relation $\LHD$ and since $\wt{v}_k \sim \avg{\wt{v}_k}$, $\wt{w}_k \sim \avg{\wt{w}_k}$  by \eqref{v average scaling} this implies 
\[
\sum_{l: l \LHD k} \avg{\wt w_k}\avg{\wt v_l} + \avg{b_k} \avg{ \wt w_k} \sim  \sum_{l:k \LHD l} \avg{\wt v_k}\avg{\wt w_l} + \avg{a_k}\avg{\wt v_k} \,.
\]
With $\avg{\wt w_k}\avg{\wt v_k} \sim 1$ from  \eqref{v average scaling} we find 
\bels{comparison relation vi}{
 \avg{\wt v_k}^2  \sim \frac{\sum_{l: l \LHD k} \avg{ \wt v_l} + \avg{b_k} }{{\sum_{l:k \LHD l} \frac{1}{\avg{\wt  v_l}} + \avg{a_k}}}
 \sim \max\cB{\max_{l \LHD k}\avg{\wt v_l}, \avg{b_k}}\min \cbb{\min_{k \LHD l} \avg{\wt v_l}, \frac{1}{\avg{a_k}}}\,.
}
Now we determine the size of $a$ and $b$. With \eqref{monotonicity for v} and \eqref{v average scaling} we have 
\[
(Qb)_k =  \frac{\tau}{\sum_{l}\wt S_{kl}\wt w_l } \sim  \tau\1 \avg{\wt v_k}\,, \qquad (Q^ta)_k = \frac{\tau}{\sum_l \wt S^t_{lk}\wt v_l} \sim \frac{\tau}{\avg{\wt v_k}}
\]
With the definition of the relation $\prec$ from \eqref{def of orders} we conclude
\bels{scaling for a and b}{
\avg{b_k} \sim \tau\max_{l \prec k} \avg{\wt v_l}\,, \qquad \frac{1}{\avg{a_k}} \sim \frac{1}{\tau}\min_{k \prec l}\avg{\wt v_l}\,.
}

Now we define 
\[
f_k(\tau) : = - \frac{\log \avg{\wt v_k(\tau)}}{\log \tau}\,.
\]
Thus,  by definition $\avg{\wt v_k} = \tau^{-f_k}$ and \eqref{power law bound} implies $\abs{f_k } \lesssim 1$. 
Now let $\tau_n \downarrow 0$ be a sequence such that $\wh{f}:= \lim_{n \to \infty}(f_k(\tau_n))_{k=1}^L$ exist. From \eqref{monotonicity for v} we get the following corollary.

\begin{corollary} \label{lmm:monotonicity}
For all $k,l \in \llbracket L \rrbracket$ the relation $k \LHD l$ implies $\wh{f}_k \le \wh{f}_l$. 
\end{corollary}

Next we derive an equation for $\wh{f}$. Taking the negative logarithm of \eqref{comparison relation vi} plugged into \eqref{scaling for a and b}, dividing by $\log \tau$ and then taking the limit along the sequence $\tau_n$ yields 
\bels{non-herm min max averaging}{
\wh{f}_k =\frac{1}{2}\pbb{
\max\cbb{\max_{l \LHD k}{\wh{f}_l} ,  \max_{l\prec k}{\wh{f}_l}-1}+ \min \cbb{ \min_{k \LHD l}{\wh{f}_l},\min_{k \prec l}{\wh{f}_l}+1}} \,, \qquad k \in \llbracket L \rrbracket\,.
}

This min-max equation \eqref{non-herm min max averaging} is similar to the one studied in \cite{Kruegersingularity}. The main difference between the two setups is the presence of terms involving the $\prec$-relation and the lack of a boundary condition for \eqref{non-herm min max averaging}. Due to the  lack of  boundary terms that were present in \cite{Kruegersingularity} the solution to \eqref{non-herm min max averaging} can only be determined up to an additive constant for all the entries. Even with the additional condition $\langle v \rangle = \langle w \rangle$, which, from the relation $\wt v_k \sim \frac{1}{\wt w_k}$, implies $\max_{k}\wh{f}_k = - \min_{k}\wh{f}_k$ equation \eqref{non-herm min max averaging} does not in general have a unique solution. Fortunately, in order to determine the growth rate of the singularity in the density $\sigma$,  we only need to determine the smallest difference $\wh f_k - \wh f_l$ for $l \LHD k$.

For an edge in $e = (l,k) \in \cal{E}$, we define  its weight  to be
\bels{def:weights}{w_e :=  \begin{cases} \wh f_k - \wh f_l &\text{ if } e \in \cal{E}^{\LHD} \\
                        \wh f_k - \wh f_l +1 &\text{ if } e \in \cal{E}^{\prec}  \setminus \cal{E}^{\LHD}  
                        \end{cases} }

\begin{lemma} \label{lem:smaller neighbor}
For every edge  $ (l,k) \in \cal{E}$, there exists an edge $(k,p)\in \cal{E}$ such that $w_{ (l,k)} \geq w_{(k,p)}$.
\end{lemma}

\begin{proof}
Multiplying the min-max averaging property, \eqref{non-herm min max averaging}, by $2$ and rearranging gives:
\bels{f equation rewritten}{ \wh f_k -   \max\cbb{\max_{q \LHD k}{\wh{f}_q} ,  \max_{q \prec k}{\wh{f}_q}-1}  = \min \cbb{ \min_{k \LHD p}{\wh{f}_p},\min_{k \prec p}{\wh{f}_p}+1} - \wh f_k \,. }
For all $ (l,k) \in \cal{E}$, we have that $w_{(l,k)} \geq \wh f_k -   \max\cb{\max_{q \LHD k}{\wh{f}_q} ,  \max_{q \prec k}{\wh{f}_q} -1} $, where the lower bound is the left hand side in \eqref{f equation rewritten}. Now we pick a minimizing $p$ on the right hand side of \eqref{f equation rewritten}. Then $(k,p) \in \cal{E}$ is the edge from the statement of the lemma.
\end{proof}

\begin{proposition}\label{prp:kappa = delta}
Let $\kappa:= \min_{\gamma} \frac{\avg{\gamma}}{\abs{\gamma}}$
with the minimum  taken over all cycles $\gamma$ and $\delta:=\min\{\wh{f}_k-\wh{f}_l: l \LHD k\}$, as well as $\wh \delta:= \min\{w_{{e}}: {e} \in \cal{E}\}$.  Then $\kappa=\delta=\wh{\delta}$.
\end{proposition}

\begin{proof}
For $K=1$ there is nothing to show. Therefore, let $K\ge 2$.       
On a cycle $\gamma = e_1 e_2 \dots e_\ell$ passing through edges $e_l =(k_{l-1},k_{l}) \in \cal{E}$ with $k_0=k_\ell$  we  have that the sum of the weights is
\begin{equation} \label{eq:sum weights}
\sum_{e\in \gamma} w_e = \langle \gamma \rangle,
\end{equation}
where we recall that $\langle \gamma \rangle$ is the number of edges in  $\cal{E}^{\prec}  \setminus \cal{E}^{\LHD}$ that $\gamma$ traverses. Thus, the minimizing $\gamma$ in the definition of  $\kappa$ has the smallest average weight. Now let $\wh \delta= \min\{w_{{e}}: {e} \in \cal{E}\}$ as in the statement of the proposition. Then there is a cycle $\gamma$ along which the weights are constant and equal to $\wh{\delta}$. Indeed, by Lemma~\ref{lem:smaller neighbor} we can iteratively construct a path $\wh \gamma =  e_1 e_2 \dots  e_k$ starting at some $ e_1$ with $w_{e_1} = \wh{\delta}$ such that $w_{ e_l} \ge w_{ e_{l+1}}$. Since $ e_1$ already has minimal weight, the weights are constant along this path. Since the graph is finite this path will contain a cycle $ \gamma$ if it is extended sufficiently far.  This cycle $ \gamma$ has constant weights along its edges equal to $\wh{\delta}$. In particular, it is a minimizer for $\kappa$ and $\kappa = \wh{\delta}$. This also implies that every minimizing cycle for $\kappa$ must have constant weights $\wh{\delta}$ along its edges. It remains to show that $\delta=\wh{\delta}$.
This holds because along a minimizing cycle $\gamma$ for $\kappa$ not all edges are in $\cal{E}^\prec \setminus \cal{E}^\LHD$. Indeed, since $S$ is irreducible  the set of edges $\cal{E}^\LHD$ is not empty. Thus, any minimizing cycle for $\kappa$ must contain at least one such edge $(l,k) \in \cal{E}^\LHD$. The weight  of this edge is $\wh{\delta} = w_{(l,k)} = \wh{f}_k- \wh{f}_l$.
\end{proof}

\begin{corollary}\label{crl:scaling problem}
In the limit $\tau\downarrow 0$, the pair $(v,w)$ satisfies the scaling problem 
\[
\lim_{\tau \downarrow 0}\frac{\log \avg{1-v Sw}}{\log \tau}  = \kappa\,. 
\]
\end{corollary}

\begin{proof}
By \eqref{vw symmetry} and \eqref{vw equation with 1 on lhs} we have 
\[
Q_1(1-vSw)=Q_1\frac{\tau  v}{ S^tv} = \frac{\tau \wt{v}}{ Q\wt{S}^t\wt{v}} \sim \frac{\tau \wt{v}}{ Q\wt{v}},
\]
where we used \eqref{wtS w sim w}. By the definition of the $\prec$-relation from \eqref{def of orders} along the subsequence $\tau_n \downarrow 0$ with $v_n:=v(\tau_n)$ and $w_n:=w(\tau_n)$ we get 
\[
\lim_{n \to \infty}\frac{\log \avg{1-v_n Sw_n}}{\log \tau_n} = \min_{l \in \llbracket L \rrbracket}\pb{1-\wh{f}_l+\min_{k: l \prec k}\wh{f}_k} =\kappa\,,
\] 
where for the last equality we used the definition of the weights in \eqref{def:weights} and that $\wh{\delta}= \kappa$ by Proposition~\ref{prp:kappa = delta}. Furthermore, we used that the minimizing cycle $\gamma$ from the proof of  Proposition~\ref{prp:kappa = delta}, for which the weights are constant and equal to $\kappa$ along $\gamma$, contains at least one edge in $\cal{E}^\prec \setminus \cal{E}^\LHD$, since $\cal{E}^\LHD$ does not contain cycles. 
\end{proof}

\cedit{ We are now ready to prove our main theorem. We will use the representation of the Brown measure, \eqref{sigma through Fr} below, from \cite{AK_Brown}. The majority of the proof is devoted to showing that $\sigma(\tau) \sim \frac{1}{\pi} \langle vw \rangle$, for $\tau$ small, then the main result follows from the relationship $
\tau v w = \frac{1-vSw}{vSw}
$
due to \eqref{vw symmetry} and \eqref{vw equation with 1 on lhs}, combined with Corollary \ref{crl:scaling problem}.  }

\begin{proof}[Proof of Theorem~\ref{thr:Singularity degree}]
We use the identity  \cite[Equation (6.17)]{AK_Brown} to  express the Brown measure at \cedit sufficiently small \nc $\tau >0$  as the quadratic form
\cedit 
\bels{sigma through Fr}{
\sigma
= \frac{1}{\pi}\scalarbb{ \sqrt{v w}}{\frac{1}{1-\sqrt{1-D\p{\hat{v}}^2}{K}_r\sqrt{1-D\p{\hat{v}}^2}} \sqrt{v w}}
}
where  the symmetric  matrices $F: \C^{2K} \to \C^{2K}$ and \nc ${K}_r: \C^K \to \C^K$ are defined as 
\cedit \bes{
K_r := P\frac{2F}{1+F} P^*\,,\qquad 
F := \begin{pmatrix}
0 & D\p{\frac{v}{\hat{v}}}SD\p{\frac{w}{\hat{v}}}
\\
D\p{\frac{w}{\hat{v}}}S^tD\p{\frac{v}{\hat{v}}} & 0
\end{pmatrix}
\,,
} 
$D(u)\in \C^{K\times K} $ is a diagonal matrix with $u$ along its main diagonal, $P: \C^{2K}\to \C^{K}$ is the projection $P(x,y):=\frac{1}{2}(x+y)$,  $P^*: \C^{K}\to \C^{2K}$ the embedding $P^*x:=(x,x)$ and  $\wh{v}:=\sqrt{v Sw}$. 

We note that the derivation of the formula \cite[Equation (6.17)]{AK_Brown} is purely algebraic. It does not rely on the Assumption {\bf A1} made in \cite{AK_Brown}.
Furthermore, in  \eqref{sigma through Fr} we took the limit $\eta \downarrow 0$, which is still explicit in \cite[Equation (6.17)]{AK_Brown}. This is possible here, since, as we will soon show, $K_r$ is well defined  and $\spec(K_r) \subset (-\infty,1]$. Furthermore, $\wh{v}$ is strictly positive at sufficiently small $\tau$.  To see that $K_r$ is well defined, i.e. that $(1+F)^{-1}$ is bounded on the image of $P^*$,   we first note that  $F(\wh{v},\wh{v})=(\wh{v},\wh{v})$ and $F(\wh{v},-\wh{v})=(-\wh{v},\wh{v})$. Since $S$ is irreducible, so is $F$. Therefore, by the Perron-Frobenius theorem, $F$ has non-degenerate eigenvalues $1$ and $-1$ and all other eigenvalues lie in the open interval $(-1,1)$. Since the eigenvector  $(\wh{v},-\wh{v})$, corresponding to the eigenvalue $-1$, is orthogonal to the image of $P^*$, the matrix $K_r$ is well defined. The inclusion $\spec(K_r) \subset (-\infty,1]$ follows because $\norm{F}\le 1$. 

Now we show that along any sequence $\tau_n \downarrow 0$ the norm $\norm{K_r}$ is bounded, i.e. we have $\limsup_{n \to \infty}\norm{K_r}\big|_{\tau=\tau_n}< \infty$. To see this we write $F$
\nc 
  in terms of $\wt{v}$ and $\wt{w}$ through
\bels{Fr in terms of wilde v and tilde w}{
D\pB{\frac{v}{\hat{v}}}SD\pB{\frac{w}{\hat{v}}} = Q_1^tD\pB{\frac{\wt{v}}{\wt{S}\wt{w}} }^{1/2}\wt{S}D\pB{\frac{\wt{w}}{\wt{S}^t\wt{v}}}^{1/2}Q_2\,.
}
Since $\frac{\wt{v}}{\wt{S}\wt{w}} \sim \wt{v}^2$ and $\frac{\wt{w}}{\wt{S}^t\wt{v}} \sim \wt{w}^2 \sim \frac{1}{\wt{v}^2}$ by Lemma~\ref{lmm:complementary scaling} we see that 
\[
\pB{D\pB{\frac{\wt{v}}{\wt{S}\wt{w}} }^{1/2}\wt{S}D\pB{\frac{\wt{w}}{\wt{S}^t\wt{v}}}^{1/2} x}_k =  \sqrt{\frac{\wt{v}_k}{\wt{S}_{kk}\wt{w}_k(1+ q_k)}}\wt{S}_{kk}\pbb{\sqrt{\frac{\wt{w}_k}{\wt{S}^t_{kk}\wt{v}_k(1 + p_k)}}x_k} + R_kx
\]
for $k\in \llbracket L \rrbracket$, 
where $q_k, p_k \in \R^{d_k}$ and  $R_k: \C^K \to \R^{d_k}$ with $\wt{S}_{kk} \in \R^{d_k \times d_k}$ satisfy 
\[
\norm{q_k} \lesssim \max_{l: k \LHD l}\frac{\avg{\wt{w}_l}}{\avg{\wt{w}_k}} \sim \max_{l: k \LHD l}\frac{\avg{\wt{v}_k}}{\avg{\wt{v}_l}}\,, \quad \norm{p_k} \lesssim \max_{l: l \LHD k}\frac{\avg{\wt{v}_l}}{\avg{\wt{v}_k}}\,,\quad 
\norm{R_k} \lesssim \max_{l: k\LHD l} \sqrt{\frac{\avg{\wt{v}_k}\avg{\wt{w}_l}}{\avg{\wt{w}_k}\avg{\wt{v}_l}}} \sim \max_{l: k\LHD l} \frac{\avg{\wt{v}_k}}{\avg{\wt{v}_l}}\,.
\]
Thus, $\norm{q_k}+ \norm{p_k} +\norm{R_k} \to 0$ for $\tau\downarrow 0$. 
\cedit 
We conclude that 
\bels{F tau limit}{
F -  
\begin{pmatrix}
 Q_1^t  & 0
\\
0& Q_2^t
\end{pmatrix}
\begin{pmatrix}
0 &  \wt F 
\\
 \wt F^t  & 0 
\end{pmatrix}
\begin{pmatrix}
 Q_1  & 0
\\
0& Q_2
\end{pmatrix} \to 0
}
\nc
for $\tau\downarrow 0$, where 
\[
(\wt{F}x)_k =  \wt{F}_kx_k\,, \qquad \wt{F}_k=D\pB{ \frac{\wt{v}_k}{\wt{S}_{kk}\wt{w}_k}}\wt{S}_{kk}\pB{\frac{\wt{w}_k}{\wt{S}^t_{kk}\wt{v}_k}} \in \R^{d_k \times d_k}.
\]
By Lemma~\ref{lmm:complementary scaling}, we have $\wt{F}_k \sim \wt{S}_{kk}$. Since $\wt{S}_{kk}$ are primitive with positive main diagonal, there is an exponent $p \in \N$ such that \cedit  $((\wt F \wt F^t)^px )_k \gtrsim \avg{x_k}$ \nc  for all $k \in  \llbracket L \rrbracket$ and $x\ge 0$.  
\cedit 
This implies that the off-diagonal $2 \times 2$ block matrix in \eqref{F tau limit} containing $\wt F$ in its $(1,2)$-block equals a direct sum $\oplus_k  \wh{F}_k$ of  $L$ irreducible off-diagonal $2 \times 2$ block matrices of the form
\[
\wh{F}_k:=
\begin{pmatrix}
0 &  \wt F_k 
\\
 \wt F_k^t  & 0 
\end{pmatrix}
\] each of which has $1$ and $-1$ as a nondegenerate  eigenvalue and all other eigenvalues in the interval $[-1+\eps,1-\eps]$ for some $\eps>0$.  
Since $F(\wh{v},\wh{v})=(\wh{v},\wh{v})$ and $\wh{v} \to 1$ for $\tau\downarrow 0$ by Corollary~\ref{crl:scaling problem}, we see that $\lim_{\tau\downarrow 0}\wt{F}_k 1 = 1$
 and $\lim_{\tau\downarrow 0}\wt{F}_k^t 1 = 1$. Now let $F_0$ be an accumulation point of $F$ as $\tau\downarrow 0$, or equivalently let  
 $\tau_n \downarrow 0$ be a sequence such that $F_0=\lim_{n \to \infty} F|_{\tau=\tau_n}$ exists along this sequence. 
 Then $F_0$ has the $L$-fold degenerate eigenvalues $1$ and $-1$ and otherwise all eigenvalues of $F_0$ lie in $[-1+\eps,1-\eps]$. The eigenvectors corresponding to the eigenvalues $\pm 1$ have the form $(Q_1^t y^{(k)},\pm Q_2^ty^{(k)})$,
 where $y^{(k)}_k:=1 \in \R^{d_k}$ and $y^{(k)}_l:=0\in \R^{d_l}$ for all $l \ne k$.  In particular, 
 \[
K_{r,0}:=\lim_{n \to \infty} K_r|_{\tau=\tau_n} = P\frac{2F_0}{1+F_0} P^*
 \]
 is a bounded matrix since  $(Q_1^t y^{(k)},- Q_2^ty^{(k)})$ does not lie in the image of $P^*$ for any $k$. This proves $\lim_{n \to \infty}\norm{K_r}\big|_{\tau=\tau_n} = \norm{K_{r,0}}< \infty$. 
\nc

Again, since  $\lim_{\tau\downarrow 0}\wh{v} = 1$, the formula \eqref{sigma through Fr} for $\sigma$ implies 
\[
\lim_{\tau\downarrow 0}\frac{\sigma(z)|_{\abs{z}^2 =\tau}}{\avg{v(\tau)w(\tau)}} = \frac{1}{\pi}\,,
\]
where the asymptotic behavior of $\avg{vw}$ is seen from 
\[
\tau v w = \frac{1-vSw}{vSw}
\]
due to \eqref{vw symmetry} and \eqref{vw equation with 1 on lhs}. Therefore, by Corollary~\ref{crl:scaling problem} again we have 
\[
-\lim_{\tau \downarrow 0}\frac{\log \avg{vw}}{\log \tau} =  \kappa-1\,.
\]
This concludes the proof of the theorem. 
\end{proof}

\begin{appendix}

\nc \section{Girko's trick}
\label{sec:Girko's trick}

\begin{proof}[Proof of Proposition~\ref{prop:Global law for A}]
The argument below follows the presentation in \cite{AK_pseudo}. 
Let $A \in \R^{K \times K} \otimes \R^{n \times n}$ be as in the statement of the proposition. {\cedit We begin by letting $X$ denote the centered, normalized version of $A$:   
\[
x_{\alpha\beta}^{ij} := \frac{1}{\sqrt{n}}\pb{a_{\alpha\beta}^{ij}-  p_{ij}}\,,
\]
in other words
\[ X = \frac{1}{\sqrt{n} } (A - \E[A]) \] }
The resolvent {\cedit $R:=R(z,\eta):=(\wh{A}_z-\ii \eta)^{-1}$ } of the Hermitization 
\[
\wh{A}_z:=\begin{pmatrix} 0 & \frac{1}{\sqrt{n}}A-z \\ \pb{\frac{1}{\sqrt{n}}A-z}^* &0  \end{pmatrix}
\]
of the rescaled adjacency matrix $\frac{1}{\sqrt{n }}A$ satisfies 
{\cedit \bels{R - G relation}{
R=G-G(\wh{P}\otimes E) \frac{1}{\frac{1}{\sqrt{n}}(\wh{P}\otimes E)+(\wh{P}\otimes E)G(\wh{P}\otimes E)}(\wh{P}\otimes E)G \,,
} }
where $G=G(z,\eta)$ is the resolvent of the Hermitization of $X$, i.e. 
\[
G=  \begin{pmatrix} \ii \eta & X-z \\ (X-z)^* & \ii \eta  \end{pmatrix}^{-1},
\]
and the inverse on the right hand side of \eqref{R - G relation} is taken on the range of  $ \wh{P} \otimes E$, where
\[
\wh{P}=\begin{pmatrix} 0 & P \\ P^t & 0  \end{pmatrix} ,
\]
 $P=(p_{ij})_{i,j=1}^K$ and $E$ is the matrix with constant entries equal to $\frac{1}{n}$. {\cedit Note that $\wh{P} \otimes E $ is the expectation of Hermitization of $A$, i.e. 
 \[\wh{P} \otimes E  = \E \begin{pmatrix} 0 & A \\ A^* & 0  \end{pmatrix}, \] and that the variances of the centered matrix $X$ are given by $s_{ij}:=\var(x_{\alpha\beta}^{ij}) = p_{ij}(1- p_{ij})$.} Now we use Girko's formula 
\begin{equation} \label{eq:Girkos formula}
\frac{1}{n} \sum_{w } f(w) -\int_{\C} f(z) \sigma(\dd z)
\frac{1}{2\pi } \int_{\C} \Delta f(z) h(z) \dd^2 z + O\Big(\frac{1}{T} \norm{\Delta f}_{L^1}\Big), \nc 
\end{equation} 
for a compactly supported $C^2$-function $f$ and $T>0$. Here,
  the sum is over $w \in \spec (\frac{1}{\sqrt{n }}A)$ and with the short hand  $\underline{R}:=( {\rm id}_2 \otimes {\rm id}_K\otimes \frac{1}{n} \tr) R$ we have 

\[
h(z) :=  \frac{1}{n} \sum_w \log |w-z|-\int_0^T \Big(\vartheta(z,\eta) -\frac{1}{1+\eta}\Big)\dd \eta.
\]
As in \cite[equation (4.7)] {AK_pseudo}we split $h$ into the contributions  $h = \sum_{i=1}^4h_i$ with 
\begin{align*} 
h_1(z) & := \int_{\eta_*}^T \bigg( \vartheta(z,\eta)  - \im \frac{1}{2K}\tr \underline{R}(z,\eta)\bigg) \dd \eta, & 
h_2(z) & := - \int_0^{\eta_*}  \im \frac{1}{2K}\tr \underline{R}(z,\eta)\dd \eta \\ 
h_3(z) & := \frac{1}{4n} \sum_{\lambda \in \spec(\wh{A}_z)} \log \bigg( 1 + \frac{\lambda^2}{T^2} \bigg) 
 - \log \bigg( 1 + \frac{1}{T} \bigg), & 
h_4(z) & := \phantom{-} \int_0^{\eta_*} \vartheta(z,\eta) \dd \eta,
\end{align*} 
\nc
where we set $\vartheta(z,\eta):=\avg{v(\abs{z}^2, \eta)}$. 
We also split $h_1 = h_{1,1}+h_{1,2}$, where
\[
h_{1,1}(z) := \int_{\eta_*}^T \bigg( \vartheta(z,\eta)  - \im \frac{1}{2K}\tr \underline{G}(z,\eta)\bigg) \dd \eta\,, \quad
h_{1,2}(z) :=\im \frac{1}{2K}\tr \int_{\eta_*}^T \big( \underline{G}(z,\eta)  - \underline{R}(z,\eta)\big) \dd \eta. 
\]
We show that uniformly, for  $z $ in a compact set, bounded away from zero, all error terms converge to zero in probability in the  limit $n\to \infty$, where we choose $\eta_*:= n^{-\eps}$ for a sufficiently small $\eps>0$ and $T:= \frac{\sqrt{n}}{\eta_*}$, proving the global law for $\frac{1}{\sqrt{n }}A$.  Since $ \eta \mapsto \vartheta(z, \eta)$ is bounded for $z\ne 0$  we see that $|h_4|=O(\eta_*)$. Since $\norm{\wh{A}_z}= O(\sqrt{n})$ with  high probability we  have $|h_3| = O(\eta_*)$.  To estimate $h_2$ we use
\[ -h_2(\zeta)  \leq C\eta_*^{1/2}\pB{\log \eta_* + \log \min_{\lambda }|\lambda|}\frac{1}{n} 
\im \tr R(z,\eta_*^{1/2}) ,\] 
where the minimum is taken over $\lambda \in \spec(\wh{A}_z)$. To see that the right hand side converges to zero, we first replace  $R$ by $G$ and estimate the difference 
\bels{estimate difference R and G}{
\norm{\tr \p{\underline{G}  - \underline{R}}} \le C\frac{1}{n}\tr G^*G  = \frac{C}{n\eta}\tr \im G \,.
}
Now we use that   $\im \frac{1}{2K}\tr \underline{G}(z,\eta) \to\vartheta(z,\eta) $ in probability, which was shown in \cite{Altinhomogeneous, Cookvarprof1}.  More precisely, we  use the quantitative bound from \cite[Theorem 2.1]{Erdos2017Correlated}, which states that for any given $\delta>0$ and $\eps=\eps(\delta)>0$ small enough we have 
\[
\sup_{z}\sup_{\eta_* \le \eta \le T }\P\pB{\Big|\im \frac{1}{2K}\tr \underline{G}(z,\eta)-\vartheta(z,\eta)\Big| \ge \frac{n^{\delta}}{n \eta^2 }} \le Cn^{-k},
\]
for all $k \in \N$ and a positive constant $C=C_{\delta,k}>0$, 
where the supremum over $z$ is over a compact set in the complex plane. 
\nc
In combination with \eqref{estimate difference R and G} this also shows that for any $\delta>0$,  $k \in \N$ and small enough $\eps>0$ the quantitative bound
\[
\P\pB{\sup_z |h_1(z)| > n^{-1+ \delta +\eps}} \le C n^{-k}
\]
holds. The integral over $\eta$ is treated here by a standard Monte Carlo sampling argument. 
\nc It remains to see that $\eta_*^{1/2}\log \min_{\lambda }|\lambda|$ from the estimate on $h_2$ converges to zero. This follows from the estimate on the smallest eigenvalue of Lemma~\ref{lmm:smallest ev} of $\wh{A}_z$ below. The integral over $z$ in \eqref{eq:Girkos formula} is again treated by Monte Carlo sampling and the proposition is proven. \nc 
\end{proof}
\nc
\begin{lemma}\label{lmm:smallest ev}
Let $B_n \in \R^{nK \times nK}$ be a deterministic matrix with $\norm{B_n} \le n^C$ for some constant $C>0$ and let $A_n$ be the adjacency matrix of the DSBM with connection probabilities $P$, where  $P$ has support. Then there are  constants $\wh C,\alpha,\beta>0$, depending only on $C$ and $P$, such that
\[
\P\pB{\norm{(A_n+B_n)^{-1}}\ge \wh C\1n^\beta} \le n^{-\alpha}
\]
hold for all $n \in \N$. 
\end{lemma}

\begin{proof}
By permuting the blocks with the permutations $Q_1$ and $Q_2$ from \eqref{def wt S} and using that applying the permutations does not change the singular values we can assume that $P$ has positive entries along its main diagonal. Now we iteratively apply Lemma~\ref{lmm:smallest ev as corr}, which is a consequence of \cite[Theorem 2.1]{tao2008}. 
\end{proof}

The following lemma, first proven in \cite{MR3403052}, gives a bound on the least singular value of a block matrix in terms of one of its blocks.

\begin{lemma}\label{lmm:smallest ev as corr}
Let $a, c_1'$ be positive constants.
Let $D$ be an $n \times n$  random matrix with i.i.d. entries and finite variance and $M$ be an $n \times n$ deterministic matrix of spectral norm at most $n^{c_1'}$. Let $A$ be an $(N-n) \times (N-n)$ matrix, $B$ be an $(N-n)\times n$ matrix, and $C$ be an $n\times (N-n)$ matrix. Assume there exist $b''$ and an event occurring with probability $1-N^{-a}$ on which $A^{-1}$  has norm bounded by $N^{b''}$ and $\|B\|, \|C\|$ have norm $O(1)$. Furthermore, assume that $D$ is independent of $A$, $B$ and $C$.

With the definition 
\[
Y := \begin{pmatrix} A & B \\ C & D \end{pmatrix}, \qquad \mathcal{M}:= \begin{pmatrix} 0 & 0 \\ 0 & M \end{pmatrix},
\]
 there are positive constants $b'$ and $c_2'$ (depending on $a$, $c_1'$ and $b''$) such that
\[ \P( \norm{( \mathcal{M}+Y)^{-1}} \ge N^{b'}) \leq c_2' N^{-a}.\]
\end{lemma}
\begin{proof}
Recall that the least singular value of a matrix $X$ can be characterized as $\|X^{-1}\|^{-1}=\min_{\|u\|=1} \|X u\|$.
Let $u$ be a vector such that {\cedit $\|(Y+\mathcal{M})u\| = \norm{(Y+\mathcal{M})^{-1}}^{-1}$} and partition $u$  and $ v:=(Y+\mathcal{M}) u $ as 
\[
u=\begin{pmatrix} u_1 \\ u_2 \end{pmatrix} ,\qquad v= \begin{pmatrix} v_1 \\ v_2 \end{pmatrix} \,,
\] where $u_1, v_1$ are $N-n$ dimensional and $u_2, v_2$ are $n$ dimensional. Expressing $v = (Y+\mathcal{M})u$ is terms of the blocks leads to \[ A u_1 + B u_2 = v_1 \,, \qquad C u_1 + (D+M) u_2 = v_2. \] If $\|u_2\| = o(N^{-b})$ then $A u_1 = v_1 + o(N^{-b})$ and by our assumption on the least singular value of $A$ we conclude, $\| v_1 \| \geq c N^{-b''}$ with probability $1 - N^{-a}$.

Otherwise, using the assumption that there is an event of probability $1-N^{-a}$ on which $A$ is invertible, we solve for $u_1$. Substituting into the second equation gives
\[ (C A^{-1} B + M + D) u_2 = v_2 -A^{-1} v_1. \]
Since $C A^{-1} B$ is independent of $D$ and has norm bounded by $N^{b''}$ with probability $1-N^{-a}$, the result \cite[Theorem 2.1]{tao2008} implies $\| v_2 -A^{-1} v_1 \| \geq n^{-b} \| u_2 \|$. Implying that $\| v \| \geq \frac{1}{2} n^{-2b-b''}$.
\end{proof}

\section{Non-negative matrices}
\label{sec:Non-negative matrices}
Here we collect a few facts and definitions concerning matrices with nonnegative entries that we use in this work. We refer to \cite{brualdi_ryser_1991} for details.

\begin{definition}Let $R = (r_{ij})_{i,j=1}^K$ be a matrix with non-negative entries. For any permutation $\pi$ of $\llbracket K \rrbracket$ we call $(r_{i\pi(i)})_{i=1}^K$ a diagonal of $R$. The diagonal with $\pi=\rm{id}$ is called main diagonal.  The matrix $P=(\delta_{i\pi(j)})_{i,j=1}^K$ is called a permutation matrix. The $0$-$1$ matrix $Z_R:=(\bbm{1}(r_{ij}>0))_{i,j=1}^K$ is called the zero pattern of $R$. 
\begin{enumerate}
\item $R$ is said to have support if it has a positive diagonal. Equivalently, $R$ has support if there is a positive constant  $c>0$ such that $R \ge c P$ holds entry wise for some permutation matrix $P$. 
\item $R$ is said to have total support if every positive entry of $R$ lies on some positive diagonal, i.e. if its zero pattern coincides with that of a sum of permutation matrices.   
\item $R$ is said to be fully indecomposable (FID) if for any index sets $I,J \subset \llbracket K \rrbracket$ with $\abs{I} + \abs{J} \ge K$ the submatrix $(r_{ij})_{i \in I,j \in J}$ is not a zero-matrix. 
\end{enumerate}
\end{definition}

The following facts about FID matrices are well known in the literature. 

\begin{lemma} 
\label{lmm:FID matrices}
Let $R \in \R^{K \times K}$ be a matrix with non-negative entries. 
\begin{enumerate}
\item If $R$ is FID and $P$ a permutation matrix, then $RP$ and $PR$ are FID.
\item 
A  matrix $R$ with non-negative entries is FID if and only if there exist a permutation matrix $ {P}$ 
such that $ {R}{P} $ is irreducible and has positive main diagonal. 
\item 
 If $ {R} $ is FID, then there is an integer $k \in \N$ such that $R^k$ has strictly positive entries,  i.e. $R$ is primitive. 
\end{enumerate}
\end{lemma}
\begin{theorem}[Frobenius-K{\"o}nig theorem]
A matrix $S \in \R^{n \times n}$ with non-negative entries does not have support if and only if $S$ contains an $r\times s$-submatrix of zeros with $r+s=n+1$.
\end{theorem}

\section{Examples}
\label{sec:Examples}

In this section we present several examples with $S$ being a $4 \times 4$ matrix. In each example the normal form $\wt S$ will be the same but the matrices $Q_1, Q_2$ used to bring $S$ into the normal form will be different, leading to different behavior in the limiting singularity at 0. From the invariance of eigenvalues by unitary conjugation it suffices to $Q_1 = I$ and thus $P = Q_2^t$ from its definition before \eqref{def of orders}. For each of these examples, we compute the $\kappa$ from \eqref{eq:defkapp}, which by Theorem \ref{thr:Singularity degree} gives the severity of the singularity near $0$. In the graphs below, the $\LHD$ edges are denoted by black, solid lines and the $\prec$ edges are denoted by red, dashed lines.

\begin{example}
Let $S$, $P$ and $SP$ be given by
\[
S = \mfour{0 & 1 & 1 &0}{0 & 0 & 1 &1}{1 & 0 & 0 &1}{1 & 0 & 0 &0},\qquad  P =  \mfour{0 & 0 & 0 &1}{1 & 0 & 0 &0}{0 & 1 & 0 &0}{0 & 0 & 1 &0},\qquad \wt S = SP =  \mfour{ 1 & 1 &0 & 0}{0 & 1 & 1 &0}{0 & 0 & 1 &1}{0 & 0 & 0 &1}. 
\]The associated graph is: 

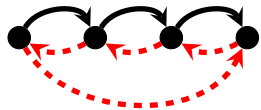
\begin{center}
\begin{tikzpicture}[scale=1, xshift=-2cm]
   
    \foreach \x in {1,...,4} {
        \node[circle, black, draw, minimum size=8.4pt, inner sep=0pt, fill=black] (v\x) at (\x,0) {};
    }

     \draw[->, bend left=45, >=stealth, black, line width=2pt] (v1) to[out=60, in=115, looseness=1] (v2);
     \draw[->, bend left=45, >=stealth, black, line width=2pt] (v2) to[out=60, in=115, looseness=1] (v3);
     \draw[->, bend left=45, >=stealth, black, line width=2pt] (v3) to[out=60, in=115, looseness=1] (v4);

       \draw[->, >=stealth, red, dashed, line width=2pt] (v2) to[out=-150, in=-30, looseness=1]  (v1);
       \draw[->, >=stealth, red, dashed, line width=2pt] (v3) to[out=-150, in=-30, looseness=1]  (v2);
       \draw[->, >=stealth, red, dashed, line width=2pt] (v4) to[out=-150, in=-30, looseness=1]  (v3);
         \draw[->, bend left=45, >=stealth, red,  dashed, line width=2pt] (v1) to[out=-60, in=-115, looseness=1] (v4);

\end{tikzpicture}
\end{center}
and $\kappa = \frac{1}{2}$.

\end{example}

\begin{example}
Let $S$, $P$ and $SP$ be 
\[
S = \mfour{0 & 0 & 1 &1}{1 & 0 & 0 &1}{1 & 1 & 0 &0}{0 & 1& 0 &0}, \qquad P =  \mfour{0 & 0 & 1 &0}{0 & 0 & 0 &1}{1 & 0 & 0 &0}{0 & 1 & 0 &0}, \qquad \wt S = SP =  \mfour{ 1 & 1 &0 & 0}{0 & 1 & 1 &0}{0 & 0 & 1 &1}{0 & 0 & 0 &1}. 
\]
The associated graph is: 

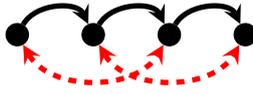
\begin{center}
\begin{tikzpicture}[scale=1, xshift=-2cm]
   
    \foreach \x in {1,...,4} {
        \node[circle, black, draw, minimum size=8.4pt, inner sep=0pt, fill=black] (v\x) at (\x,0) {};
    }

     \draw[->, bend left=45, >=stealth, black, line width=2pt] (v1) to[out=60, in=115, looseness=1] (v2);
     \draw[->, bend left=45, >=stealth, black, line width=2pt] (v2) to[out=60, in=115, looseness=1] (v3);
     \draw[->, bend left=45, >=stealth, black, line width=2pt] (v3) to[out=60, in=115, looseness=1] (v4);
     
       \draw[<->, >=stealth, red, dashed, line width=2pt] (v1) to[out=-60, in=-120, looseness=1]  (v3);
       \draw[<->, >=stealth, red, dashed, line width=2pt] (v2) to[out=-60, in=-120, looseness=1]  (v4);  
       
           
\end{tikzpicture}
\end{center}
and $\kappa = \frac{1}{3}$.
\end{example}

\begin{example}
Let $S$, $P$ and $SP$ be 
\[
{\cedit S = \mfour{0 & 0 & 1 &1}{0 & 1 & 1 &0}{1 & 1 & 0 &0}{1 & 0& 0 &0},}\qquad P =  \mfour{0 & 0 & 0 &1}{0 & 0 & 1&0}{0 & 1 & 0 &0}{1 & 0 & 0 &0}, \qquad \wt S = SP =  \mfour{ 1 & 1 &0 & 0}{0 & 1 & 1 &0}{0 & 0 & 1 &1}{0 & 0 & 0 &1}.
\]
 The associated graph is: 

\begin{center}
\begin{tikzpicture}[scale=1, xshift=-2cm]
   
    \foreach \x in {1,...,4} {
        \node[circle, black, draw, minimum size=8.4pt, inner sep=0pt, fill=black] (v\x) at (\x,0) {};
    }

     \draw[->, bend left=45, >=stealth, black, line width=2pt] (v1) to[out=60, in=115, looseness=1] (v2);
     \draw[->, bend left=45, >=stealth, black, line width=2pt] (v2) to[out=60, in=115, looseness=1] (v3);
     \draw[->, bend left=45, >=stealth, black, line width=2pt] (v3) to[out=60, in=115, looseness=1] (v4);
     
       \draw[<->, >=stealth, red, dashed, line width=2pt] (v1) to[out=-60, in=-120, looseness=1]  (v4);
       \draw[<->, >=stealth, red, dashed, line width=2pt] (v2) to[out=-30, in=-150, looseness=1]  (v3);

\end{tikzpicture}
\end{center}
and $\kappa = \frac{1}{4}$.
\end{example}

\end{appendix}

\providecommand{\bysame}{\leavevmode\hbox to3em{\hrulefill}\thinspace}
\providecommand{\MR}{\relax\ifhmode\unskip\space\fi MR }
\providecommand{\MRhref}[2]{%
  \href{http://www.ams.org/mathscinet-getitem?mr=#1}{#2}
}
\providecommand{\href}[2]{#2}

\end{document}